\documentclass[11pt,english]{smfart}
   
\usepackage{dsfont,eucal,mathptmx,txfonts}

\usepackage[text={6.5in,9in},centering]{geometry}

\newcommand{\inv}{^{\raisebox{.2ex}{$\scriptscriptstyle-1$}}}

\usepackage[dvipsnames]{xcolor}   
\usepackage{xparse}
\usepackage{xr-hyper}
\definecolor{brightmaroon}{rgb}{0.76, 0.13, 0.28}
\usepackage[linktocpage=true,colorlinks=true,hyperindex,citecolor=teal,linkcolor=blue]{hyperref}

\renewcommand{\theequation}{\thesection.\arabic{equation}}

\newtheorem{thm}{Theorem}[section]
\newtheorem{exam}[thm]{Example}
\newtheorem{defn}[thm]{Definition}

\newtheorem{prop}[thm]{Proposition}
 
\newtheorem{lem}[thm]{Lemma}
\newtheorem{cor}[thm]{Corollary}

\newtheorem{rem}[thm]{Remark}
\renewcommand{\labelenumi}{\upshape(\arabic{enumi})}

\numberwithin{equation}{section}

\author{A. Goswami}
\address{[1] Department of Mathematics and Applied Mathematics, University of Johannesburg, P.O. Box 524, Auckland Park 2006, South Africa. [2] National Institute for Theoretical and Computational Sciences (NITheCS), South Africa.} 
\email{agoswami@uj.ac.za}

\title{Normal structure spaces of groups}


\begin{document}
  
\begin{abstract}
Can we do a topological study of various classes of normal subgroups endowed with a hull-kernel-type topology? In this paper, we have provided an answer to this question. We have introduced as well a new class of normal subgroups called primitive subgroups. Separation axioms, compactness, connectedness, and continuities of these spaces have been studied. We have concluded with the question of determining spectral spaces among them.
\end{abstract}   

\maketitle 

\section{Introduction} 
In \cite{S58}, the striking similarity has been pointed out between the properties of ideals of a commutative Noetherian ring and normal subgroups of a group with the maximal condition. This analogy has been achieved by the following correspondences: 
sum of ideals corresponds to the product of normal subgroups;
product of ideals corresponds to the commutator of the normal subgroups; and the residual quotient of ideals has an analogue introduced in \cite{S58}. Following the lead of \cite{S58}, the author in \cite{K64} has asked what else can be said in the presence of a finiteness condition, and  has introduced notions of prime and primary ideals, the radical of an ideal, $m$-systems, isolated components and residual quotients. The main result in \cite{K64} is to obtain existence and uniqueness theorems for the decomposition of a given normal subgroup as an intersection of primary subgroups.  Further studies have been done in \cite{ZY00, R04, ZL11, LZ15} by considering radical properties as a tool in investigating the structure of a group. In \cite{S60, S68}, prime subgroups are investigated in the
context of groups satisfying the maximal condition on normal subgroups. A study
of prime subgroups of groups have also recently been done in \cite{FGT22}.

Since the fundamental paper \cite{S37} on topologized (known as Stone topology) maximal ideals of Boolean algebras, intensive study of hull-kernel topologies have been done on various classes of ideals. Primitive ideals (of a noncommutative ring) endowed with a Jacobson topology (known as structure space) has been considered in \cite{J45} (see also \cite{J56}). The same Jacobson topology on primitive ideals of an enveloping algebra plays a crucial role in representations of complex finite-dimensional Lie algebras (see \cite{D96, J83, J95}). 
For the construction of an affine scheme, in \cite{G60}, it (called the Zariski topology or the spectral topology) has been considered on prime ideals of a commutative ring. In \cite{HJ65} (see also \cite{H71}), the hull-kernel topologies have been endowed on minimal prime ideals of a commutative ring. In a more recent paper \cite{A08} strongly irreducible ideals of a commutative ring endowed with the hull-kernel topology have been studied. In \cite{DG22}, a study has been made on various classes of ideals of a (commutative) ring  endowed with coarse lower topologies. Since we can not impose hull-kernel topologies on an arbitrary class of ideals, we need to classify such classes, where coarse lower topologies also coincide with the hull-kernel ones, and all these have been obtained in \cite{DG22}. 

We shall here consider prime, minimal prime, maximal, primary, regular, principal, and finitely generated normal subgroups of a group. We shall introduce primitive normal subgroups (see Definition \ref{prns}) which best in author's knowledge has never been considered before. Following the strategies from \cite{DG22}, we shall endow coarse lower topologies on all these classes of normal subgroups and shall study various topological properties of these ``normal structure spaces'' (see Definition \ref{nss}).    

\section{Preliminaries} 
In this paper $G$ will denote a group and $e$ the identity element of $G$. We shall omit the binary operation $\cdot$ of $G$ and shall write $a\cdot b$ as $ab.$ The set of all normal subgroups of $G$ shall be denoted by $\mathrm{N}_G$ and the proper normal subgroups by $\mathrm{Prp}_G$. We shall denote the trivial (normal) subgroup of $G$ by $\mathfrak{e}$. If $\{\mathfrak{n}_{\lambda}\}_{\lambda \in \Omega}$ is a collection of normal subgroups of $G$, then we shall denote the join (\textit{i.e.}, the normal subgroup generated by $\bigcup_{\lambda \in \Omega}\mathfrak{n}_{\lambda}$) of them by   $\sum_{\lambda \in \Omega}\mathfrak{n}_{\lambda},$ which is also a normal subgroup of $G$. If $\mathfrak{n}, \mathfrak{n}'\in \mathrm{N}_G,$ then $[\mathfrak{n}, \mathfrak{n}']$ will denote the subgroup generated by the commutators
$$[xyx\inv y\inv\mid x\in \mathfrak{n}, y\in \mathfrak{n}'].$$
Taking product of two normal subgroups $\mathfrak{n}, \mathfrak{n}'$ of $G$ as $[\mathfrak{n}, \mathfrak{n}']$, it is easy to verify that $(\mathrm{N}_G,[\,,],\cap)$ is a multiplicative lattice (see \cite{FFJ22}), from which we have the following important inclusion relation:
\begin{equation}\label{psi}
[\mathfrak{n}, \mathfrak{n}']\subseteq \mathfrak{n}\cap \mathfrak{n}'.
\end{equation}

Here we briefly talk about various classes of normal subgroups, on which later on we endow coarse lower topologies. 

\begin{defn}\cite[p.\,377]{S58}.
\emph{A proper normal subgroup $\mathfrak{p}$ of $G$ is called \emph{prime} whenever  $[\mathfrak{n}, \mathfrak{n}']\subseteq \mathfrak{p}$ implies either $\mathfrak{n}\subseteq \mathfrak{p}$ or $\mathfrak{n}'\subseteq \mathfrak{p}.$ We shall denote the set of all prime normal subgroups of $G$ by $\mathrm{Spec}_G$.  }
\end{defn} 

The following result characterizes groups in which trivial normal subgroups are prime.

\begin{prop}\cite[Lemma 2.1]{FGT22}.\label{afc}
The following are equivalent for a group $G$:
 
\begin{enumerate}
\item The trivial subgroup of $G$ is prime in $G.$
\item The lattice $\mathrm{N}_G$ is uniform (that is, the intersection of any two nontrival normal subgroups of G is nontrivial) and the only normal abelian subgroup of G is the trivial subgroup.
\end{enumerate}
\end{prop}

An equivalent definition of a prime normal subgroup follows from the following 
\begin{prop}\cite[Proposition 3.4]{K64}.
In a group $G$ the following are equivalent.
\begin{enumerate}
\item $\mathfrak{p}$ is a prime normal subgroup of $G$;
\item $[\langle a\rangle, \langle b \rangle]\subseteq \mathfrak{p}$ implies either $\langle a\rangle \subseteq \mathfrak{p}$ or $\langle b\rangle \subseteq \mathfrak{p}$ for all $a, b\in G.$
\end{enumerate}
\end{prop}

A prime normal subgroup $\mathfrak{p}$ is called  \emph{minimal prime}  belonging to a normal subgroup $\mathfrak{n}$, if it contains $\mathfrak{n}$, and if there is no prime normal subgroup containing $\mathfrak{n}$ which is strictly contained in $\mathfrak{p}$. We shall denote the set of all minimal prime normal subgroups of $G$ by $\mathrm{MinSpec}.$ 
Following \cite[Proposition 1.13]{K64}, we say that the \emph{radical} $\sqrt{\mathfrak{n}}$ of a normal subgroup $\mathfrak{n}$ in $G$ is
the intersection of all prime normal subgroups containing $\mathfrak{n}$. Note that this definition of radical is somewhat different from given in \cite[p. 376]{S58}. If $\sqrt{\mathfrak{n}}=\mathfrak{n},$ then
$\mathfrak{n}$ is called a \emph{radical normal subgroup} (or simply \emph{radical}) and we denote set of all such normal subgroups of $G$ by $\mathrm{Rad}_G$. Like in rings, the radical of a normal subgroup also enjoy the following properties. 
 
\begin{prop}\cite[Proposition 1.14]{K64}.
Suppose that $\mathfrak{n}, \mathfrak{n}'\in \mathrm{N}_G.$ Then
\begin{enumerate}
\item $\sqrt{\mathfrak{n}}\supseteq \mathfrak{n},$
\item $\mathfrak{n}\supseteq \mathfrak{n}'$ implies $\sqrt{\mathfrak{n}}\supseteq \sqrt{\mathfrak{n}'},$
\item $\sqrt{\sqrt{\mathfrak{n}}}=\sqrt{\mathfrak{n}},$
\item $\sqrt{[\mathfrak{n}, \mathfrak{n}']}=\sqrt{\mathfrak{n}\cap \mathfrak{n}'}=\sqrt{\mathfrak{n}}\cap \sqrt{\mathfrak{n}'}.$
\end{enumerate}
\end{prop}

The following result characterizes prime normal subgroups in terms of radical and irreducible normal subgroups. 

\begin{prop}\cite[p.\,208]{K64}. 
A normal subgroup $\mathfrak{p}$ of a group $G$ is prime if and
only if it is radical and  irreducible.
\end{prop}

\begin{defn}\cite[Proposition 2.2]{K64}.
\emph{A normal subgroup $\mathfrak{q}$ of a group $G$ is called \emph{primary} if $[\mathfrak{n}, \mathfrak{n}']\subseteq \mathfrak{q}$ and $\mathfrak{n}\nsubseteq \mathfrak{q},$ then $\mathfrak{n}'\subseteq \sqrt{\mathfrak{q}}.$ We denote the set of all primary normal subgroups of $G$ by $\mathrm{Prim}_G$.}
\end{defn}

\begin{prop}\cite{K64}.
If $\mathfrak{q}$ is radical then the following conditions are equivalent:
\begin{enumerate}
\item $\mathfrak{q}$ is a prime subgroup of $G$,
\item $\mathfrak{q}$ is a primary subgroup of $G$,
\item $\mathfrak{q}$ is irreducible.
\end{enumerate}
\end{prop}

\begin{defn}
\emph{A normal subgroup $\mathfrak{n}$ of $G$ is \emph{strongly irreducible}, if whenever $\mathfrak{n}\subseteq \mathfrak{a}\cap \mathfrak{b},$ where $\mathfrak{a}, \mathfrak{b}\in \mathrm{N}_G,$ then either $\mathfrak{n}\subseteq\mathfrak{a}$ or $\mathfrak{n}\subseteq\mathfrak{b}$. We shall denote the set of all strongly irreducible normal subgroups of $G$ by $\mathrm{Irr}^+_G.$ }
\end{defn}

\begin{rem}
\emph{Strongly irreducible ideals were introduced in \cite{F49} for commutative rings, under the name  \emph{primitive} ideals (not to be confused with Jacobson's primitive ideals). In \cite[p.\,301, Exercise 34]{B72}, the ideals of the same spectrum are called \emph{quasi-prime} ideals. The terminology ``strongly irreducible'' was first used for noncommutative rings in \cite{B53}.
These subgroups are going to play a crucial role in `determining' hull-kernel topologies as we will see in next section. }
\end{rem}

A normal subgroup of a group $G$ is termed a \emph{finitely generated} if it is finitely generated as a group, and we denote the set of such subgroups by $\mathrm{Fgen}_G$. The set of all principal normal subgroups (\textit{i.e.}, the subgroup generated by a single element of $G$) is denoted by $\mathrm{Prin}_G$. A non-trivial normal subgroup $\mathfrak{n}$ of a group $G$ is called \emph{minimal} such that between $\mathfrak{n}$ and the trivial subgroup $1$ there are no other normal subgroups of $G$ and we denote the set of all such normal subgroups by $\mathrm{Min}_G$. 

\begin{defn}
\emph{We say a proper normal subgroup $\mathfrak{m}$ of a group $G$ is \emph{maximal} if there is no proper normal subgroup of $G$ that properly contains $\mathfrak{m}$. We shall denote the set of all maximal normal subgroups of $G$ by $\mathrm{Max}_G$.} 
\end{defn}

\begin{rem}\label{nmns}
\emph{Note that any nontrivial finite group has maximal normal
subgroups. Moreover, $G/\mathfrak{m}$ is a simple group if and only if $\mathfrak{m}$ is a
maximal normal subgroup of $G$ (see \cite{AB95}).   
However, there are infinite groups with no maximal normal subgroups. Consider a Pr\"{u}fer group: $$\mathbb{Z}_{p^{\infty}}=\lim_{\longleftarrow}\mathbb{Z}_{p^{i}}\qquad (p\;\text{is a prime}),$$
which is an abelian group and hence all subgroups are normal, but none of them is maximal. This is in contrast with unitary rings (\textit{i.e.}, rings with identity), where assuming axiom of choice always guarantees existence of maximal ideals. }
\end{rem}

Following \cite{KS04}, next we define regular normal subgroups. The action of a group $G$ on a set $\Omega$ is \emph{regular} if, for every pair $(\lambda, \beta)\in \Omega\times \Omega,$ there exists exactly one $g\in G$ such that $\lambda^g=g\inv \lambda g=\beta.$

\begin{defn}
\emph{If $\mathfrak{n}$ is a normal subgroup of $G$
that acts regularly on $\Omega$, then $\mathfrak{n}$ is called a \emph{regular} normal subgroup of $G.$ We denote the set of all regular normal subgroups of $G$ by $\mathrm{Reg}_G$. } 
\end{defn}

Finally we introduce the notion of a primitive normal subgroup of a group. A (left) \emph{$G$-module} is an abelian group $(M,+,0)$ endowed with a map $G\times M\to M$ (defined by: $(g,m)\mapsto gm$) satisfying the following identities:
\begin{enumerate}
\itemsep -.1em
\item[(i)] $g(m+m')=gm+gm',$
\item[(ii)] $(gg')m=g(g'm),$
\item[(iii)] $em=m,$
\item[(iv)] $g0=0,$
\end{enumerate}
for all $g, g'\in G$ and for all $m,m'\in M.$ A $G$-module is said to be \emph{simple} if $M$ does not have any submodule except $0$ and itself. A (left) \emph{annihilator} of $G$ is defined by
$$\mathrm{Ann}_G(M)=\{g\in G\mid gm=0\;\text{for all}\; m\in M\},$$
where $M$ is a simple $G$-module. Note that an annihilator of $G$ can be defined for any $G$-module $M$. However, we have considered simple module $M$ because eventually we will use it to define a primitive normal subgroup of $G$. Clearly,

\begin{prop}
An annihilator $\mathrm{Ann}_G(M)$ is a normal subgroup of $G$.
\end{prop}

%\begin{proof} Note that $(gng\inv)m=(gn)(g\inv m)=g(n(g\inv m))=g0=0,$ for all $g\in G,$ $n\in \mathrm{Ann}_G(M),$ and this implies $gng\inv \in \mathrm{Ann}_G(M)$ as required. \end{proof}

\begin{defn}\label{prns}
\emph{A normal subgroup $\mathfrak{p}$ of a group $G$ is called \emph{primitive} if there exists a simple $G$-module such that $\mathfrak{p}=\mathrm{Ann}_G(M).$ We shall denote the set of all primitive normal subgroups of $G$ by $\mathrm{Pmtv}_G$. }
\end{defn}

\begin{rem}
\emph{Note that the definition above is in alignment with the definition of a primitive ideal of a (noncommutative) ring, where instead of $G$-module, we take usual notion of a module over a ring (see \cite{J45, J56}).} 
\end{rem}  
   
Having fixed any of the above defined properties for the normal
subgroups of a group $G$, we define $\sigma_G$ as the set of all normal subgroups of
$G$ satisfying this property. We call $\sigma_G$ a \emph{spectrum} of $G$.  We assume $G\notin\sigma_G$ for all spectra of $G$. By a \emph{subspectrum} of a spectrum $\sigma_G$, we mean a spectrum that is also a subset of $\sigma_G$. 

\section{Normal structure spaces}

The coarse lower topology (also known as lower topology \cite{G et. al.}) on a spectrum $\sigma_G$ is the topology for which the sets of the type 
$$
\mathcal{V}(\mathfrak x)=\{\mathfrak n\in \sigma_G\mid \mathfrak x\subseteq \mathfrak n \}
\qquad (\mathfrak x\in \mathrm{N}_G)$$
form a subbasis of closed sets. Moreover, we define an (algebraic) \emph{closure} operator on a spectrum $\sigma_G$ as follows:
$$\mathrm{Cl}(X)=\{\mathfrak n\in \sigma_G\mid \cap X\subseteq \mathfrak n \}\qquad (X\subseteq \mathrm{N}_G).$$

\begin{defn}\label{nss}
\emph{A spectrum $\sigma_G$  endowed with a coarse lower topology will be called a \emph{normal structure space}}
\end{defn}

We use the  same notation $\sigma_G$ to denote the space. In general, the collection of subsets $\{\mathcal{V}(\mathfrak{x})\}_{\mathfrak{x}\in \mathrm{N}_G}$ is not closed under finite unions. However the following result characterizes a spectrum $\sigma_G$ on which each member of the collection is indeed a closed set, \textit{i.e.}, $\{\mathcal{V}(\mathfrak{x})\}_{\mathfrak{x}\in \mathrm{N}_G}$ induces a hull-kernel topology on such a $\sigma_G$. A ring-theoretic version of this result was first proved in \cite[p.\,11]{M53}.  

\begin{thm}
\label{hkt}
The collection of sets $\{\mathcal{V}(\mathfrak{x})\}_{\mathfrak{x}\in \mathrm{N}_G}$   induces a hull-kernel topology on a spectrum $\sigma_G$ if and only if $\sigma_G$ has the following property:    
\begin{equation}
\mathfrak{n}\cap \mathfrak{n}'\subseteq \mathfrak{s}\quad \text{implies}\quad  \mathfrak{n}\subseteq \mathfrak{s}\;\; \text{or}\;\; \mathfrak{n}'\subseteq \mathfrak{s} \label{mip}
\end{equation}
for all $\mathfrak{n},$ $\mathfrak{n}'\in \mathrm{N}_G$ and for all $\mathfrak{s}\in  \sigma_G.$  
\end{thm}
 
The `largest' spectrum that satisfies the condition \eqref{mip} is  nothing but $\mathrm{Irr}^+_G,$ the spectrum of strongly irreducible normal subgroups. From (\ref{psi}) it follows that $\mathrm{Spec}_G$ and $\mathrm{MinSpec}_G$ also satisfy the condition \eqref{mip}, and hence these are the examples of spectra on which we can define hull-kernel  topology. These facts summarizes in the following
 
\begin{cor}   
If $\sigma_G$ is $\mathrm{Irr}^+_{G}$ or its subspectra, then the coarse lower topology on it coincide with the corresponding hull-kernel topology.
\end{cor}
 
Here are some examples of spectra that do not satisfy the condition (\ref{mip}) and hence one can not impose hull-kernel topologies on them. 

\begin{exam}
\emph{Suppose that  $G=(\mathbb{Z},+)$ and consider $\mathfrak{n}=2\mathbb{Z},$ $\mathfrak{n}'=3\mathbb{Z},$ and $\mathfrak{s}=6\mathbb{Z}.$ Now $2\mathbb{Z}\cap 3\mathbb{Z}=6\mathbb{Z}\subseteq 6\mathbb{Z},$ but $2\mathbb{Z}\nsubseteq 6\mathbb{Z}$ and $3\mathbb{Z}\nsubseteq 6\mathbb{Z}.$ This shows that the spectra $\mathrm{Prop}_G$, $\mathrm{Prin}_G$ and $\mathrm{Fgen}_G$ do not satisfy the condition (\ref{mip}).  }
\end{exam}
 
We now study various topological properties of normal structure spaces. Some of the proofs of these results are similar to the corresponding results of rings, whereas some require additional assumptions or even do not hold at all in the context of groups. As we pass on we will highlight on this comparison with rings.  
We start with continuous maps between normal structure spaces. Observe that although inverse image of a normal subgroup under a group homomorphism is normal, but the same may not hold for an arbitrary spectra $\sigma_G$. To resolve this problem, we need to impose that property on a spectrum. Hence we have the following

\begin{defn}
\emph{We say a spectrum $\sigma_G$ satisfies the \emph{contraction} property if for any group homomorphism $\phi\colon G\to G',$ the inverse image  $\phi\inv(\mathfrak{n}')$ is in $\sigma_G$, whenever $\mathfrak{n}'$ is in $\sigma_{G'}.$ }
\end{defn}

It has been shown in \cite[Lemma 2.2]{FGT22} that $\mathrm{Spec}_G$ has the contraction property.
Since the sets $\{\mathcal{V}(\mathfrak{x})\}_{\mathfrak{x}\in \mathrm{N}_G}$ only form a (closed) subbasis, all our arguments need to be at this level rather than just closed sets.

\begin{prop}\label{conmap}
Let $\sigma_G$ be a spectrum satisfying the contraction property. Let $\phi\colon G\to G'$ be a group homomorphism  and $\mathfrak{n}'\in\sigma_{G'}.$ 
\begin{enumerate}
\item\label{contxr} The induced map $\phi_*\colon  \sigma_{G'}\to \sigma_G$ defined by  $\phi_*(\mathfrak{n}')=\phi\inv(\mathfrak{n}')$ is    continuous.

\item If $\phi$ is  surjective, then the normal structure space $\sigma_{G'}$ is homeomorphic to the closed subspace $\mathcal{V}(\mathrm{Ker}(\phi))$ of the normal structure space $\sigma_G.$

\item\label{den} The subset  $\phi_*(\sigma_{G'})$ is dense in $\sigma_G$ if and only if $\mathrm{Ker}(\phi)\subseteq \bigcap_{\mathfrak{s}\in \sigma_G}\mathfrak{s}.$ 
\end{enumerate}
\end{prop}
   
\begin{proof}      
To show (i), let $\mathfrak{x}\in \mathrm{N}_G$ and $\mathcal{V}(\mathfrak{x})$ be a   subbasic closed set of the normal structure  space $\sigma_G.$ Then  
\begin{align*}
(\phi_*)\inv(\mathcal{V}(\mathfrak{x})) &=\{ \mathfrak{n}'\in  \sigma_{G'}\mid \phi\inv(\mathfrak{n}')\in \mathcal{V}(\mathfrak{x})\}\\&=\{\mathfrak{n}'\in \sigma_{G'}\mid \phi(\mathfrak{x})\subseteq \mathfrak{n}'\}\\&=\mathcal{V}(\langle\phi(\mathfrak{x})\rangle), 
\end{align*} 
and hence the map $\phi_*$  continuous.     
For (ii), observe that $\mathrm{Ker}(\phi)\subseteq \phi\inv(\mathfrak{n}')$ follows from the fact that  $\mathfrak{e}\subseteq \mathfrak{n}'$ for all $\mathfrak{n}'\in \sigma_{G'}.$ It can thus been seen that $\phi_*(\mathfrak{n}')\in \mathcal{V}(\mathrm{Ker}(\phi)),$ and hence $\mathrm{Im}(\phi_*)=\mathcal{V}(\mathrm{Ker}(\phi)).$  
If $\mathfrak{n}'\in \sigma_{G'},$ then
\begin{align*}
\phi(\phi_*(\mathfrak{n}'))&=\phi(\phi\inv(\mathfrak{n}'))=\mathfrak{n}'\cap G'=\mathfrak{n}'.
\end{align*}
Thus $\phi_*$ is injective. To show that $\phi_*$ is closed, first we observe that for any   subbasic closed subset  $\mathcal{V}(\mathfrak{a})$ of  $\sigma_{G'}$, we have
\begin{align*}
\phi_*(\mathcal{V}(\mathfrak{x}))&=  \phi\inv(\mathcal{V}(\mathfrak{x}))\\&=\phi\inv\{ \mathfrak{n}'\in \sigma_{G'}\mid \mathfrak{x}\subseteq   \mathfrak{n}'\}\\&=\mathcal{V}(\phi\inv(\mathfrak{x})). 
\end{align*}
Now if $K$ is a closed subset of $\sigma_{G'}$ and $K=\bigcap_{ \lambda \in \Omega} (\bigcup_{ i \,= 1}^{ n_{\lambda}} \mathcal{V}(\mathfrak{x}_{ i\lambda})),$ then
\begin{align*}
\phi_*(K)&=\phi\inv \left(\bigcap_{ \lambda \in \Omega} \left(\bigcup_{ i = 1}^{ n_{\lambda}} \mathcal{V}(\mathfrak{x}_{ i\lambda})\right)\right)\\&=\bigcap_{ \lambda \in \Omega} \bigcup_{ i = 1}^{x_{\lambda}} \mathcal{V}(\phi\inv(\mathfrak{x}_{ i\lambda}))
\end{align*}
a closed subset of  $\sigma_G.$ Since by (\ref{contxr}), $\phi_*$ is continuous, we have the desired claim.
Finally to prove (iii), first we wish to show: $\mathrm{Cl}(\phi_*(\mathcal{V}(\mathfrak{n}')))=\mathcal{V}(\phi\inv(\mathfrak{n}'))$ for all $\mathfrak{n}'\in \mathrm{N}_{G'}.$ For that, let $\mathfrak{s}\in \phi_*(\mathcal{V}(\mathfrak{n}')).$ This implies $\phi(\mathfrak{s})\in \mathcal{V}(\mathfrak{n}'),$ and that means $\mathfrak{n}'\subseteq \phi(\mathfrak{s}).$ Therefore, $\mathfrak{s}\in \mathcal{V}(\phi\inv(\mathfrak{n}')).$ Since $\phi\inv(\mathcal{V}(\mathfrak{n}'))=\mathcal{V}(\phi\inv(\mathfrak{n}'))$, the other inclusion follows. If we take $\mathfrak{n}'$ as the trivial normal subgroup $\mathfrak{e}'$ of $G'$, the above identity becomes
$$\mathrm{Cl}(\phi_*(\sigma_{G'}))=\phi_*(\mathcal{V}(\mathfrak{e}))=\mathcal{V}(\phi\inv(\mathfrak{e}))=\mathcal{V}(\mathrm{Ker}(\phi)),$$ and hence  $\mathcal{V}(\mathrm{Ker}(\phi))$ to be equal to $\sigma_G$ if and only if $\mathrm{Ker}(\phi)\subseteq \bigcap_{\mathfrak{s}\in \sigma_G}\mathfrak{s}.$ 
\end{proof}  

For $\mathrm{Spec}_G$, the inclusion condition in (\ref{den}) is replaced by an equality.  If $\phi$ is the quotient map $G\to G/\mathfrak{n}$, then we have the following

\begin{cor}
If $\sigma_G$ is a spectrum of $G$ satisfying the contraction property, then the normal structure space $\sigma_{{G/\mathfrak{n}}}$ is homeomorphic to the closed subspace
$\mathcal{V}(\mathfrak{n})$ of $\sigma_G$ for every $\mathfrak{n}\in \mathrm{N}_G$.   
\end{cor} 

Let us recall that a nonempty closed subset $K$ of a topological space $X$ is \emph{irreducible} if $K\neq K_{\scriptscriptstyle 1}\cup K_{\scriptscriptstyle 2}$ for any two proper closed subsets  $K_{\scriptscriptstyle 1}, K_{\scriptscriptstyle 2}$ of $K$. A point $x$ in a closed subset $K$ is called a \emph{generic point} of $K$ if $K = \mathrm{Cl}(x).$  

\begin{thm}\label{irrc}  
In every normal structure space $\sigma_G$, the subbasic closed sets of the form: $$\{\mathcal{V}(\mathfrak{n})\mid \mathfrak{n}\in \mathcal{V}(\mathfrak{n})\}$$ are irreducible. 
\end{thm} 

\begin{proof} 
It is sufficient to show that $\mathcal{V}(\mathfrak{n})=\mathrm{Cl}(\mathfrak{n})$ whenever $\mathfrak{n}\in\mathcal{V}(\mathfrak{n})$. Since $\mathrm{Cl}(\mathfrak{n})$ is the smallest closed set containing $\mathfrak{n}$, and since $\mathcal{V}(\mathfrak{n})$ is a closed set containing $\mathfrak{n}$, obviously then  $\mathrm{Cl}(\mathfrak{n})\subseteq \mathcal{V}(\mathfrak{n})$. 
For the reverse inclusion, if $\mathrm{Cl}(\mathfrak{n})= \sigma_G$, then 
\[ 
\sigma_G=\mathrm{Cl}(\mathfrak{n})\subseteq \mathcal{V}(\mathfrak{n})\subseteq \sigma_G.
\] 
This proves that $\mathcal{V}(\mathfrak{n})=\mathrm{Cl}(\mathfrak{n})$. Suppose that $\mathrm{Cl}(\mathfrak{n})\neq \sigma_G$. Since $\mathrm{Cl}(\mathfrak{n})$ is a closed set,  there exists an  index set, $\Omega$, such that,  for each $\lambda\in\Omega$, there is a positive integer $x_{\lambda}$ and normal subgroups $\mathfrak{x}_{\lambda 1},\dots, \mathfrak{x}_{\lambda n_\lambda}$ of $G$ such that 
$$
\mathrm{Cl}(\mathfrak{n})={\bigcap_{\lambda\in\Omega}}\left({\bigcup_{ i\,=1}^{ x_\lambda}}\mathcal{V}(\mathfrak{x}_{\lambda i})\right).
$$
Since  
$\mathrm{Cl}(\mathfrak{n})\neq \sigma_G,$ we can assume that ${\bigcup_{ i\,=1}^{ x_\lambda}}\mathcal{V}(\mathfrak{x}_{\lambda i})$ is non-empty for each $\lambda$. Therefore, $\mathfrak{n}\in   {\bigcup_{ i\,=1}^{ x_\lambda}}\mathcal{V}(\mathfrak{x}_{\lambda i})$ for each $\lambda$, and hence $\mathcal{V}(\mathfrak{n})\subseteq {\bigcup_{ i=1}^{ x_\lambda}}\mathcal{V}(\mathfrak{x}_{\lambda i})$, that is $\mathcal{V}(\mathfrak{n})\subseteq \mathrm{Cl}(\mathfrak{n})$ as desired. 
\end{proof}     

\begin{cor}\label{spiir}
Every non-empty subbasic closed subset of $\mathrm{Prp}_G$ is irreducible.
\end{cor} 

\begin{proof}
Since for every non-empty subbasic closed subset $\mathcal{V}(\mathfrak{n})$ of $\mathrm{Prp}_G$, the normal subgroup $\mathfrak{n}$ is also in $\mathrm{Prp}_G,$ the proof now follows from Theorem \ref{irrc}.  
\end{proof} 

As hull-kernel spaces are sparse with closed sets, so are normal structure spaces. Nevertheless, we have the following separation axiom for all of them. 
 
\begin{thm}
 Every normal structure space $\sigma_G$ is   $T_{ 0}.$ 
\label{ct0t1}  
\end{thm}

\begin{proof} 
Let $\mathfrak{n}$ and $\mathfrak{n}'$ be two distinct elements of $\sigma_G$. Then, without loss of generality, we may assume that $\mathfrak{n}\nsubseteq \mathfrak{n}'$. Therefore $\mathcal{V}(\mathfrak{n})$ is a closed set containing $\mathfrak{n}$ and missing $\mathfrak{n}'$. Therefore $\sigma_G$ is a $T_0$-space.
\end{proof} 

Let $R$ be a ring. It is well known that prime spectrum $\mathrm{Spec}(R)$ endowed with with a Zariski topology is a $T_1$-space  if and only if $\mathrm{Spec}(R)$ coincides with the spectrum $\mathrm{Max}(R)$ of maximal ideals. For a normal structure space $\sigma_G$ to be $T_1$, the condition is slightly different.
 
\begin{thm}
Suppose that $G$ is a group in which $\mathrm{Max}_G$ exists. Then a normal structure space $\sigma_G$ is a $T_1$-space if and only if $\sigma_G\subseteq\mathrm{Max}_G$. 
\end{thm}  
 
\begin{proof}  
Suppose that $\mathfrak{n}\in \sigma_G$. Then $\mathfrak{n}\in\mathcal{V}(\mathfrak{n})$, and so, by Theorem \ref{irrc}, $\mathrm{Cl}(\mathfrak{n})=\mathcal{V}(\mathfrak{n})$. Let $\mathfrak{m}$ be a maximal normal subgroup with $\mathfrak{n}\subseteq\mathfrak{m}$. Then   $\mathfrak{m}\in 	\mathcal{V}(\mathfrak{n})=\mathrm{Cl}(\mathfrak{n}) = \{\mathfrak{n}\}$, the last part because $\sigma_G$ is $T_{ 1}$. Therefore $\mathfrak{m}=\mathfrak{n}$, proving $\sigma_G\subseteq \mathrm{Max}(G)$.   
Conversely, in $\mathrm{Max}(G)$, $\mathcal{V}(\mathfrak{m})=\{\mathfrak{m}\}$ for every $\mathfrak{m}\in \mathrm{Max}(G)$, so that $\mathfrak{m}\in \mathcal{V}(\mathfrak{m})$, and hence, by Theorem \ref{irrc}, $\mathrm{Cl}(\mathfrak{m})=\{\mathfrak{m}\}$, showing that the structure space $\sigma_G$ is $T_{ 1}$.
\end{proof} 

\begin{cor}
 Let $G$ be a Noetherian group. If $\sigma_G$ is a discrete space then $G$ is Artinian.  
\end{cor} 

Another consequence of Theorem \ref{irrc} is the following sufficient condition for a structure space to be connected. 
 
\begin{thm}\label{conis}
If a spectrum $\sigma_G$ contains $\mathfrak{e},$ then the  structure space $\sigma_G$ is connected.
\end{thm}  

\begin{proof}
We observe that  $\sigma_G=\mathcal{V}(\mathfrak{e})$ and irreducibility implies connectedness. Now the claim follows from Theorem \ref{irrc}. 
\end{proof}
 
\begin{exam}
\emph{The normal structure spaces $\mathrm{Prp}_G$, $\mathrm{Fgen}_G$, $\mathrm{Prin}_G$ are connected. If $G$ is a group such that the lattice $\mathrm{N}_G$ is uniform, then by Proposition \ref{afc} and Theorem \ref{irrc}, the normal structure space $\mathrm{Spec}_G$ is connected.}
\end{exam}  

The next definition is going to play an important role in proving existence of generic points of nonempty irreducible closed sets of a normal structure space.  

\begin{defn} \label{grad} 
\emph{For a spectrum $\sigma_G$ and for $\mathfrak{x}\in\mathrm{N}_G,$ we define the normal subgroup $$\sqrt{\mathfrak{x}}^{\omega}=\bigcap_{\substack{\mathfrak{x}\subseteq \mathfrak{n} \\ \mathfrak{n}\in \sigma_G}}  \mathcal{V}(\mathfrak{x}).$$}
\end{defn}  

\begin{lem}\label{hrx} 
Let $\sigma_G$ is
a spectrum of $G$. Then
\begin{enumerate}
\label{rxlr}  
\item\label{axaa} for every $\mathfrak{x}\in \mathrm{N}_G,$ $\mathfrak{x}\subseteq \sqrt{\mathfrak x}^{\omega}.$ 

\item \label{axa}   
If $\mathfrak{n}\in\sigma_G,$ then $\mathfrak{n}=\sqrt{\mathfrak n}^{\omega}$.

\item \label{hahxa} For all $\mathfrak{x} \in \mathrm{N}_G,$ $\mathcal{V}(\mathfrak{x})=\mathcal{V}(\sqrt{\mathfrak x}^{\omega}).$

\item\label{xbxa} For every $\mathfrak{x}, \mathfrak{y} \in \mathrm{N}_G,$ the inclusion $\mathcal{V}(\mathfrak{x})\subseteq \mathcal{V}(\mathfrak{y})$ holds if and only if  $\sqrt{\mathfrak y}^{\omega}\subseteq\sqrt{\mathfrak x}^{\omega}$.
\end{enumerate}
\end{lem}

\begin{proof}
The inclusion in (i) follows from Definition \ref{grad}. 
When $\mathfrak{n}\in\sigma_G,$ then we have $$\sqrt{\mathfrak n}^{\omega}=\bigcap_{\mathfrak{n}\subseteq \mathfrak{s}}\left\{\mathfrak{s}\in\sigma_G\right\}=\mathfrak{n},$$ and that proves (ii).  
For (iii), by (\ref{axaa}), $\mathfrak{x}\subseteq \sqrt{\mathfrak x}^{\omega},$ and hence we have $\mathcal{V}(\mathfrak{x})\supseteq \mathcal{V}(\sqrt{\mathfrak x}^{\omega}).$ The other inclusion follows from Definition \ref{grad}. 
The necessity part of (iv) is obvious. Now suppose $\sqrt[ \mathrm{X}]{\mathfrak{y}}\subseteq \sqrt{\mathfrak x}^{\omega}$, and  $\mathfrak{s}\in \mathcal{V}(\mathfrak{x})$. Then $\mathfrak{s}\in \sigma_G$ and $\mathfrak{x}\subseteq \mathfrak{s}$, so that $\sqrt{\mathfrak x}^{\omega}\subseteq \mathfrak{s}$. Therefore, 
$$
\mathfrak{s}\in \mathcal{V}(\mathfrak{s})\subseteq \mathcal{V}(\sqrt{\mathfrak x}^{\omega})\subseteq \mathcal{V}(\sqrt[\mathrm{X}]{\mathfrak{y}})\subseteq \mathcal{V}(\mathfrak{y}),
$$
showing that $\mathcal{V}(\mathfrak{x})\subseteq \mathcal{V}(\mathfrak{y})$.
\end{proof}

\begin{thm}\label{sob}  
Every nonempty irreducible closed subset $\mathcal{V}(\mathfrak{n})$ of a normal structure space $\sigma_G$ has a unique generic point if and only if  $\mathcal{V}(\mathfrak{n})$  contains $\sqrt{\mathfrak n}^{\omega}.$ 
\end{thm}  

\begin{proof}
Suppose that every nonempty irreducible closed subset of $\sigma_G$ has a unique generic point and $\mathcal{V}(\mathfrak{n})$ is one such. Then there exists $\mathfrak{n}'\in \sigma_G$ such that 
$\mathcal{V}(\mathfrak{n})=\mathrm{Cl}(\mathfrak{n}')=\mathcal{V}(\mathfrak{n}').$ Then by Lemma \ref{rxlr}(\ref{axa}) we have $\mathfrak{n}'=\sqrt{\mathfrak n}^{\omega}\in \sigma_G$.     
Conversely, let every non-empty irreducible subbasic  closed set $\mathcal{V}(\mathfrak{n})$ contains $\sqrt{\mathfrak n}^{\omega}.$ and let $K$ be an irreducible closed subset of $\sigma_G$. By definition, $$K=\bigcap_{\lambda\in \Omega}E_{\lambda},$$ where $E_{\lambda}$ is a finite union of sets of the type $\mathcal{V}(\mathfrak{n})$, for suitable normal subgroups $\mathfrak n$ of $G$. Since $K$ is irreducible, for every ${\lambda}\in \Omega$ there exists an normal subgroup $\mathfrak n_{\lambda}$ of $G$ such that $K\subseteq \mathcal{V}(\mathfrak{n}_{\lambda})\subseteq E_{\lambda}$ and thus, if $\mathfrak t= \sum_{{\lambda}\in \Omega}\mathfrak n_{\lambda}$, then we have  $$K=\bigcap_{{\lambda}\in \Omega}\mathcal{V}(\mathfrak n_{\lambda})=\mathcal{V}(\mathfrak t)=\mathcal{V}(\sqrt{\mathfrak t}^{\omega}).$$ By assumption, $\sqrt{\mathfrak t}^{\omega}\in \sigma_G$ and thus $K=\mathrm{Cl}(\sqrt{\mathfrak t}^{\omega})$. 
\end{proof}
  
\begin{prop}\label{sirrs} 
If $G$ has maximal normal subgroups, then every nonempty irreducible closed subset of $\mathrm{Irr^+}_G$ has a unique generic point.
\end{prop}
   
\begin{proof} 
Since $\mathrm{N}_G$ is a multiplicative lattice with product as  intersection, every strongly irreducible normal subgroup becomes a prime element of this lattice and hence by \cite[Lemma 2.6]{FFJ22}, every nonempty irreducible closed subset of $\mathrm{Irr^+}_G$ has a unique generic point.
\end{proof}

\begin{rem} 
\emph{Since every non-empty irreducible closed subset $\mathcal{V}(\mathfrak{n})$ of $\mathrm{Prp}_G$ contains $\mathfrak{n},$ every such set has a unique generic point. }
\end{rem}

For a commutative ring $R$, it is well known that spectrum of prime ideals $\mathrm{Spec}(R)$ is compact with respect to Zariski topology. The same is also true for any class of ideals that contains all the maximal ideals of $R$. For groups the situation is more subtle because of possible nonexistence of maximal normal subgroups in a group. However, having the following property guarantees compactness of certain normal structure spaces.

\begin{defn}\label{fulle}
\emph{We say a normal structure space has the \emph{partition of unity} property if for every $\mathfrak{n}\in \mathrm{N}_G,$ $\mathcal{V}(\mathfrak{n})=\emptyset$ implies $\mathfrak{n}=G.$}
\end{defn}  

\begin{prop}
If a structure space $\sigma_G$ have the partition of unity property then $\sigma_G$ contains all maximal normal subgroups of $G$ whenever they exist.
\end{prop}
	
\begin{proof}
Suppose $\sigma_G$ has the partition of unity property. Let $\mathfrak{m}$ be a maximal normal subgroup of $G$. If $\mathfrak{m}$ were not in $\sigma_G$, then we would have $\mathcal{V}(\mathfrak{m}) = \emptyset$, and since $\mathfrak{m} \neq G$, we would have a contradiction.
\end{proof} 
  
\begin{rem}
\emph{For rings the converse also holds, whereas in case of groups, maximal normal subgroups may not exist (see Remark \ref{nmns}), and hence application's of Zorn's lemma fails.}
\end{rem}

\begin{thm}\label{comp} 
If a normal structure space $\sigma_G$ has the partition of unity property, then $\sigma_G$ is compact. 
\end{thm} 
  
\begin{proof}    
Let  $\{K_{ \lambda}\}_{\lambda \in \Omega}$ be a family of subbasic closed sets of a normal structure space $\sigma_G$   such that $\bigcap_{\lambda\in \Omega}K_{ \lambda}=\emptyset.$ Let $\{\mathfrak{n}_{ \lambda}\}_{\lambda \in \Omega}$ be a family of normal subgroups of $\mathrm{N}_G$ such  that $\forall \lambda \in \Omega,$  $K_{ \lambda}=\mathcal{V}(\mathfrak{n}_{ \lambda}).$  Since $\bigcap_{\lambda \in \Omega}\mathcal{V}(\mathfrak{n}_{ \lambda})=\mathcal{V}\left(\sum_{\lambda \in \Omega}\mathfrak{n}_{ \lambda}\right),$ we get  $\mathcal{V}\left(\sum_{\lambda \in \Omega}\mathfrak{n}_{ \lambda}\right)=\emptyset,$ and that by Definition  \ref{fulle} gives $ \sum_{\lambda \in \Omega}\mathfrak{n}_{ \lambda}=G.$ Then, in particular, we obtain $e=\sum_{\lambda_i\in \Omega}a_{ \lambda_i},$ where $a_{ \lambda_i}\in \mathfrak{n}_{\lambda_i}$ and $a_{ \lambda_i}\neq 0$ for $i=1, \ldots, n$. This implies    $G=\sum_{  i \, =1}^{ n}\mathfrak{n}_{\lambda_i}.$ Hence,   $\bigcap_{ i\,=1}^{ n}K_{ \lambda_i}=\emptyset,$ and $\sigma_G$ is compact by Alexander subbasis Theorem.    
\end{proof}  

If $G$ is a Noetherian group, then we have the following result which may considered as a group-theoretic version of the corresponding result proved in \cite{FGS22} for rings. 
 
\begin{thm}
If $G$ is a Noetherian group then every normal structure space $\sigma_G$ is compact.
\end{thm}

\begin{proof}
Consider a collection $\{\sigma_G\cap \mathcal{V}(\mathfrak n_{\lambda}) \}_{\lambda\in \Omega}$ of subbasic closed sets of $\sigma_G$ with the finite intersection property. By assumption, the normal subgroup $\mathfrak  s:=\sum_{\lambda\in \Omega}\mathfrak n_{\lambda}$ is finitely generated, say $\mathfrak s=(\alpha_1,\ldots,\alpha_n)$. For  every $1\leqslant j\leqslant n$, there exists a finite subset $\Lambda_j$ of $\Omega$ such that $\alpha_j\in \sum_{\lambda\in \Lambda_j}\mathfrak n_{\lambda}$. Thus, if  $\Lambda:=\bigcup_{j=1}^n\Lambda_j$, it immediately follows that $\mathfrak s=\sum_{\lambda\in \Lambda}\mathfrak n_{\lambda}$. Hence we have
\begin{align*} 
\bigcap_{\lambda\in \Omega}(\sigma_G\cap \mathcal{V}(\mathfrak n_{\lambda}))&=\sigma_G\cap \mathcal{V}(\mathfrak s)\\&=\sigma_G\cap \mathcal{V}\left(\sum_{\lambda\in \Lambda}\mathfrak n_{\lambda} \right)\\&= \bigcap_{\lambda\in \Lambda}(\sigma_G\cap \mathcal{V}(\mathfrak n_{\lambda})) \neq \emptyset,
\end{align*}
since $\Lambda$ is finite and $\{\sigma_G\cap \mathcal{V}(\mathfrak n_{\lambda}) \}_{\lambda\in \Omega}$ has the finite intersection property. Then the conclusion follows by the Alexander's subbasis Theorem. 
\end{proof}

\begin{cor}
If $G$ is a Noetherian group then every normal structure space is  Noetherian.
\end{cor} 

\begin{rem}
\emph{It is well known that every Noetherian topological
space can be written as a finite union of non-empty irreducible closed subsets. 
The spectrum $\mathrm{Prp}(G)$ of a group $G$ always has this property irrespective of $G$ being a Noetherian group, and hence $\mathrm{Prp}(G)$ is always a Noetherian normal structure space. This follows from the fact that $\mathcal{V}(\mathfrak{e})$ is an irreducible closed subset in $\mathrm{Prp}(G)$ (see Corollary \ref{spiir}).  }
\end{rem} 

When we study topological spaces whose underlying sets are various classes of ideals, determining which ones are spectral spaces in the sense of \cite{H69} (see also \cite{DST19}) is one of the most important questions. The following result completely characterizes that of $\mathrm{Spec}_G$.

\begin{thm}\cite[Proposition 3.9]{FGT22}.
The spectrum $\mathrm{Spec}_G$ endowed with Zariski topology (= coarse lower topology) is spectral if and only if it is compact, if and
only if $G$ has a maximal normal subgroup.
\end{thm}

In \cite{FGS22}, an attempt has been made in studying spectrality of various classes of ideals (of a commutative ring) endowed with coarse lower topologies. Similarly, it will be interesting to classify normal structure spaces that are spectral. 

\section*{Acknowledgement}
The author would like to thank the anonymous referee for his/her careful reading and for several remarks that helped to
improve the presentation of the paper.

%\section*{Compliance with Ethical Standards} The following are not applicable with respect to this article. \begin{enumerate}\item Disclosure of potential conflicts of interest.\item Research involving human participants and/or animals.\item Informed consent.\end{enumerate}

\end{document}